\documentclass[12pt]{l4dc2021} 
\usepackage{enumerate,comment}
\usepackage{algorithm,algorithmicx,algpseudocode}

\newcommand{\bE}{\mathbb{E}}

\newcommand{\cP}{\mathcal{P}}
\newcommand{\eps}{\epsilon}
\newcommand{\cX}{\mathcal{X}}
\newcommand{\cU}{\mathcal{U}}

\newcommand{\bP}{\mathbb{P}}

\newcommand{\cG}{\mathcal{G}}
\newcommand{\cR}{\mathcal{R}}
\newcommand{\te}{\theta}
\newcommand{\cK}{\mathcal{K}}
\newcommand{\cC}{\mathcal{C}}
\newcommand{\cE}{\mathcal{E}}
\newcommand{\ust}{^{\star}}

\newcommand{\nal}[1]{\begin{align*}#1\end{align*}}
\newcommand{\ali}[1]{\begin{align}#1\end{align}}
\newtheorem{assumption}{\textbf{Assumption}}

\title[]{Reward Biased Maximum Likelihood Estimation for Learning in Constrained MDPs}
\usepackage{times}

\author{%
	\Name{Rahul Singh} \Email{rahulsingh@iisc.ac.in}\\
	\addr Indian Institute of Science, Bangalore, India.
}

\begin{document}
	
	\maketitle

\begin{abstract}
	We use the Reward Biased Maximum Likelihood Estimation (RBMLE) algorithm~\cite{techreport} to learn optimal policies for constrained Markov Decision Processes (CMDPs). We analyze the learning regrets of RBMLE.
\end{abstract}	
\section{Introduction} \label{Introduction}

Consider a controlled Markov chain with finite state space $X$, finite action set $U$, and controlled transition
probabilities $\bP(x(t+1)=y~|x(t)=x, u(t)=u) = p(x,y,u)$, where $x(t) \in X$ denotes the state at time $t$, and $u(t) \in U$ denotes the action
taken at time $t$. A reward $r(x,u)$ is received when action $u$ is taken in state $x$.

Let $\bar{r}(\pi,p)$ and $\bar{c}(\pi,p)$ denote the average reward and cost under a stationary policy $\pi$ when the true value of MDP parameter is $p$,
$$
\bar{r}(\pi,p) := \liminf_{T\to\infty} \frac{1}{T}\sum_{t=1}^{T} \bE_{\pi} r(x(t),u(t)).
$$
Let $R\ust(p)$ denote the maximal long-term average reward of the solution to the following CMDP when the parameter $p$ is known:
\nal{
\max_{\pi}\liminf_{T\to\infty} \frac{1}{T}\sum_{t=1}^{T} \bE_{\pi} r(x(t),u(t))\\
\mbox{ s.t. } \limsup_{T\to\infty} \frac{1}{T}\sum_{t=1}^{T} \bE_{\pi} c(x(t),u(t)) \le c_{ub}.
}
We consider the case where the transition probabilities $p$ are only known to belong a set $\Theta$, but otherwise unknown. 

A stationary randomized policy $\pi$ maps each state $x$ to a probability distribution over the space of actions $U$, i.e., $\pi(x): X \mapsto \Delta(U)$. 

\begin{lemma}
Let $\pi\ust$ be an optimal stationary randomized policy for the following CMDP
\ali{
	\max_{\pi} \bar{r}(\pi)  \mbox{ s.t. } \bar{c}(\pi) \le c_{ub}.\label{cmdp}
}
Let $\hat{\pi}$ be the policy from $\Pi_{F}$ that is ``closest'' to $\pi\ust$ (each $\pi\ust(x)\in \Delta(U)$ is mapped to the closest point in the tesselation). Then we have that
\ali{
\bar{r}(\hat{\pi}) \ge r\ust -  \epsilon \kappa_p |X|^2, \mbox{ and  }  \bar{c}(\hat{\pi}) \le c_{ub} +  \epsilon \kappa_p |X|^2.\label{ineq:lemma1_1}
}
Furthermore, the ``discretization error" 
$$
r\ust - \max_{ \pi \in \Pi_{F}: \bar{c}(\pi) \le c_{ub}    } \bar{r}(\pi),
$$
is $O(\epsilon)$.
\end{lemma}
\begin{proof}
~\eqref{ineq:lemma1_1} follows directly from Lemma~\ref{lemma:cho_meyer}. To prove the result on discretization error, construct a modified form of the CMDP~\eqref{cmdp} in which the bound on cost is replaced by $c_{ub}-  \epsilon \kappa_p |X|^2$.  Using Theorem~\ref{th:perturbation_cmdp} we conclude that the optimal reward of this modified CMDP is lower bounded by $r\ust - \epsilon \kappa_p |X|^2 \frac{\hat{\eta}}{\eta} $. Consider the discretization to this solution. It follows from~\eqref{ineq:lemma1_1} that this discretized solution is feasible for the  problem
$$
 \max_{ \pi \in \Pi_{F}: \bar{c}(\pi) \le c_{ub}    } \bar{r}(\pi),
$$
while its reward is lower bounded by $r\ust - \epsilon \kappa_p |X|^2 \frac{\hat{\eta}}{\eta} - \epsilon \kappa_p |X|^2$ (denote this by $y$). Proof is then completed by noting that the optimal discretized solution to the problem 
$$
\max_{ \pi \in \Pi_{F}: \bar{c}(\pi) \le c_{ub}    } \bar{r}(\pi)
$$
has a reward which is necessarily greater than $y$.
\end{proof}

\section{System Model}
We consider the MDP described in Section~\ref{Introduction}, assuming,
without loss of generality, that $r(x,u),c(x,u) \in [0,1]$ for all $ (x,u) \in X \times U$.
We denote by $\Pi_{SD}$ the set of all stationary deterministic policies that map $X$ into $U$,
by $\Pi_{S}$ the set of all stationary possibly randomized policies. In general, solving a CMDP requires randomization. Note that the set of stationary randomized policies is uncountably infinite. Thus, while deriving regret bounds, we will assume that we restrict ourselves to a finite set of policies $\Pi_{F}$. $\Pi_F$ could be obtained by discretizing the simplex $\Delta(U)$. Thus in each state $x$, a policy from the set $\Pi_F$ will pick an element from this finite set, and randomize amongst the actions using this distribution. Denote
$$
\Pi\ust(\te) := \arg \max_{\pi \in \Pi_{F}} \left\{\bar{r}(\te,\pi), \mbox{ s.t. }\bar{c}(\te,\pi) \le c_{ub}\right\},
$$ 
and let $r\ust(p)$ the optimal value.

\begin{definition}(Unichain MDP)
	Under a stationary policy $\pi$, let
	$\tau^{\pi}_{x,y}$ denote the time taken to hit the state $y$ when started in state $x$. The MDP is called unichain if $\bE [\tau^{\pi}_{x,y}]$ is finite for all $(x,y,\pi)$. 
\end{definition}

\begin{definition}(Mixing Time)\label{def:cond}
	For a unichain MDP with parameter $p$, its mixing time $T_{p}$ is defined as 
	\begin{align*}
		T_p := \max_{\pi \in \Pi_{s}}  \max_{x,y\in X} \bE [\tau^{\pi}_{x,y}].
	\end{align*}
	Its ``conductivity'' $\kappa_p$ is 
	\begin{align*}
		\kappa_p := \max_{\pi \in \Pi_s}  \max_{x\in X}\frac{\max_{y\neq x} \bE [\tau^{\pi}_{x,y}]  }{2\bE [\tau^{\pi}_{x,x}] }.     
	\end{align*}
\end{definition}
\begin{definition}(Kullback-Leibler divergence) For $p_2=\{p_2(x)\}$ absolutely continuous with respect to $p_2=\{p_2(x)\}$ the KL-divergence between them is	
	\begin{align}
		KL(p_1,p_2):=\sum_{x \in X} p_1(x) \log \frac{p_1(x)}{p_2(x)}.
	\end{align}    
\end{definition}

\begin{definition}(Gaps)
The reward gap is defined as follows,
	$$
	\Delta_{\min,r} :=   R\ust(p)  -  \max_{\{\pi \in \Pi_F : \bar{c}(\pi,p)\le c_{ub}\}} \bar{r}(\pi,p).
	$$
	The cost gap is defined similarly as follows,
	$$
\Delta_{\min,c} :=   \min_{\left\{\pi \in \Pi_F : \bar{c}(\pi,p) > c_{ub}\right\}} \bar{c}(\pi,p) - c_{ub}.
$$		
$	\Delta_{\min,r}$ is the reward gap, while $\Delta_{\min,c}$ is the cost gap.
$\blacksquare$
\end{definition}	

For two integers $x,y$, we use $[x,y]$ to denote the set $\{x,x+1,\ldots,y\}$ and for $a,b\in\mathbb{R}$ we let $a\vee b := \max\{a,b\}$.

\begin{assumption}\label{assum:1}
	We assume that the following information is known about the unknown transition probabilities $p(x,y,u)$ :
	\begin{itemize}
		\item the set of tuples $(x,y,u)$ for which $p(x,y,u)=0$, 
		\item a lower bound $p_{\min}$ on the non-zero transition probabilities,
		\begin{align}\label{def:p_min}
			p_{\min} := \min\limits_{(x,y,u): p(x,y,u)>0} p(x,y,u).
		\end{align}
	\end{itemize} 
	\rm We let $\Theta$ be the following set 
	\begin{align}
		& \Theta:=\notag \\
		\Bigg\{ &\te \in [0,1]^{|X| \times |X| \times |U|}: \te(x,y,u)=0\mbox{ if  }~p(x,y,u)=0, \sum_{y \in X} \te(x,y,u) =1
		\;\forall\; (x,u), \te(x,y,u)\ge 0 \Bigg\}. \label{def:Theta}
	\end{align}
	We occasionally refer to $\te \in \Theta$ as a ``parameter" describing the model or transition probabilites.
$\blacksquare$
\end{assumption}

Learning Regrets: The work~\cite{singh2020learning} shows that in order to judge the performance of a learning algorithm $\phi$ designed to solve the CMDP~\eqref{cmdp} when the controlled transition probabilities $p$ are unknown, an appropriate metric is the tuple consisting of the following two regrets: a) Reward regret $\cR_{r}(T;\phi)$, and b) Cost Regret, $\cR_{c}(T;\phi)$,
\ali{
	\cR_r(T) :&= r\ust(p)T -  \sum_{t=1}^{T}  r(x(t),u(t)),\\
	\cR_c(T) :&=  \sum_{t=1}^{T}  c(x(t),u(t)) - T (c_{ub} + \Delta_{\min,c}).
}
$	\cR_r(T)$ keeps track of the sub-optimality in reward collection, while $	\cR_c(T)$ measures the amount of ``over-utilization" of the control cost during the learning phase. Note that the definitions of regrets imply that our benchmark is $\Pi_F$. In what follows, we denote $\tilde{c}_{ub}:= c_{ub}+\Delta_{\min,c}$. a recent work on learning in constrained MDPs is~\cite{efroni2020exploration}.
\section{The RBMLE-Based Learning Algorithm}\label{sec:rbmle}
If $\te =  \{\te(x,y,u): (x,y,u) \} \in \Theta$ denotes an MDP parameter (controlled transition probabilities), denote by $\te(x,u)$ the vector $\left\{\te(x,y,u):y\in X\right\}$.
Let $n(x,u;t)$ be the number of times an action $u$ has been applied in state $x$ until time $t$, and by $n(x,y,u;t)$ the number of $x\to y$ one-step transitions under the application of action $u$. Let $\hat{p}(t)=\left\{ \hat{p}(x,y,u;t) : x,y \in X, u \in U \right\}$ be the empirical estimate of $p$ at time $t$, with $\hat{p}(x,y,u;t)$ the MLE of $p(x,y,u)$ at time $t$,
\begin{align}\label{def:empirical}
	\hat{p}(x,y,u;t) : =\frac{n(x,y,u;t)}{n(x,u;t)\vee 1} \;\forall\; x,y \in X \text{ and } u \in U .
\end{align} 
{\bf{The RBMLE algorithm:}} 
The algorithm evolves in an episodic manner. We let $\tau_k$ denote the starting time of the $k$-th episode, and $\cE_k :=[\tau_k,\tau_{k+1} -1] $ the set of time-slots that comprise it. The episode durations increase exponentially with episode index, with $|\cE_k|=2^k$. Clearly $\tau_k = \sum_{\ell = 1}^{k-1}|\cE_\ell|$. Throughout we abbreviate $\hat{p}(\tau_k)$ as $\hat{p}_{k}$, $n(x,y;\tau_k)$ as  $n_{k}(x,y)$ and $n(x,y,u;\tau_k)$ as  $n_{k}(x,y,u)$.
At the beginning of each episode $\cE_k$, the RBMLE algorithm determines the following quantities\footnote{Throughout the paper, a pre-specified priority order is used to choose a particular maximizer in $arg\max$ if needed.}:
\begin{enumerate}[(i)]
	\item  A ``reward-biased MLE'' of the true system parameter, which is denoted $\te_k$:
	\begin{align}
		&\te_k \in \arg\max_{\te \in \Theta} \bigg\{ \max_{\pi \in \Pi_{S}: \bar{c}(\te,\pi) \le c_{ub}}\bigg\{ \alpha(\tau_k) \bar{r}(\theta,\pi) -\sum_{(x,u)} n_k(x,u)KL\left( \hat{p}_{k}(x,u),\te(x,u)  \right) \bigg\} \bigg\}, \label{eq:rbmle_2}
		\end{align}
where 		
		\begin{align}
\alpha(t) :=a\log\left( t^b |X|^{2} |U|\right),\label{eq:alpha}
	\end{align}
and the constants $a,b$ are chosen by the agent and must satisfy the following conditions,
\ali{
b>2, \text{ and } 
a>\frac{|X|^3|U|}{2p_{\min} \Delta_{\min}}.\label{cond:a,b}
}
\item A stationary (possibly randomized) policy 
\ali{
	\pi_k \in  \underset{\pi \in \Pi_{F}:  \bar{c}(\te_k,\pi) \le c_{ub}}{\arg\max}\bar{r}(\te_k,\pi).\label{def:pi_k}
}
\end{enumerate}
Thereafter, it chooses controls for the current episode, i.e. for times $t \in \cE_k$ by sampling $u(t)$ from the distribution $\pi_k (x(t))$.

\textbf{Intuitive Description of RBMLE}: The RBMLE algorithm takes an optimistic view towards solving the problem by giving a ``push" or bias towards parameters with higher average rewards. As we will show, this allows it to explore the parameter space adequately, so that its ``learning regret" is low. However, since it has to also ensure that the time-average cost expenditures are below $c_{ub}$ (with a high probability), it also enforces the constraint $\bar{c}(\te,\pi)\le c_{ub}$.

%
%

\section{Preliminary Results}
{\bf{An equivalent Index description of the RBMLE learning algorithm:}}
We begin by showing that the RBMLE algorithm is an ``index rule"~\citep{gittins2011multi}. Though this characterization does not yield us simple computational procedures to implement the RBMLE algorithm, it is used while analyzing its learning regret.
\begin{lemma}\label{lemma:rbmle_index}
At the beginning of each episode $\cE_k$, RBMLE attaches an index $I_k(\pi)$ to each $\pi\in\Pi_{S}$,
\begin{align}\label{eq:index}    
	I_k(\pi):=\max_{\te \in \Theta: \bar{c}(\pi,\te)\le c_{ub}  }\bigg\{\alpha(\tau_k) \bar{r}(\theta,\pi) -\sum_{(x,u)} n_k(x,u)KL\left( \hat{p}_{k}(x,u),\te(x,u)  \right) \bigg\}.
\end{align}
In case $\bar{c}(\pi,\te)>c_{ub}$ for all $\te\in \Theta$, we let $I_k(\pi)= - \infty$. The policy $\pi_k$, as defined in~\eqref{def:pi_k} has the largest index, i.e.,
\begin{align}
	\pi_k \in \arg\max_{\pi \in \Pi_{S}} I_k(\pi). 
\end{align}
Within $\cE_k$, RBMLE implements the policy with the highest value of the RBMLE index. $\blacksquare$
\end{lemma}
 Define the following ``confidence interval'' $\cC(t)$ at time $t$ associated with the empirical estimate $\hat{p}(t)$,
	\begin{align}\label{def:d1}
	\cC(t): =
	\bigg\{ \te \in \Theta: |\te(x,y,u)-\hat{p}(x,y,u;t) |  \le d_1(x,u;t),\forall (x,y,u) \in X \times X \times U \bigg\},
	\end{align}
	where
	\begin{align}\label{eq:d1}
	d_1(x,u;t):= \sqrt{\frac{ \log\left( t^b |X|^{2} |U|\right) }{n(x,u;t)}  },
	\end{align}
and $b$ is as in~\eqref{cond:a,b}. Define also the set $\cG_1$,
\begin{align}\label{eq:goodset}
	\cG_1 := \left\{\omega:  p \in \cC(t), \; \forall\; t \in \mathbb{N} \right\}.
\end{align}

	

The following result shows concentration of $\hat{p}(t)$ around the true MDP's parameter $p$.
\begin{lemma}\label{lemma:confidence}
		The probability that $p$ lies in $\cC(t)$ is bounded as follows:
						\begin{align*}
			\bP(p \in \cC(t))> 1-\frac{2}{t^{2b-1}|X|^2|U|},\;\forall\;  t \in \mathbb{N}  .
		\end{align*}

	\end{lemma}

For a stationary policy $ \pi\in \Pi_{S}$, define $\theta_{k,\pi}$ as follows
\begin{align}\label{eq:index_theta}    
	\theta_{k,\pi} \in \arg \max_{\te \in \Theta: \bar{c}(\pi,\te)\le c_{ub}}\bigg\{\alpha(\tau_k) \bar{r}(\theta,\pi) -\sum_{(x,u)} n_k(x,u)KL\left( \hat{p}_{k}(x,u),\te(x,u)  \right) \bigg\}.
\end{align} 

We now derive a lower bound on the index of any 
optimal stationary policy $\pi \ust \in \Pi \ust (p)$ that holds with
high probability. 
\begin{lemma}\label{lemma:suff_prec_optimal}
	On the set $\cG_1$, the index $I_k(\pi\ust)$ of any optimal policy $\pi \ust \in \Pi\ust(p)$ is lower bounded as follows:
	\begin{align*}
		I_k(\pi\ust) \geq  \alpha(\tau_k)(1-\gamma)~R\ust(p), ~\forall k =1,2,\ldots,K,
	\end{align*}
	where 
	\ali{
		1 \ge \gamma \ge \frac{|X|^3|U|}{2a\cdot p_{\min} J \ust(p)}.\label{cond_gamma}
	} 
\end{lemma}
For a policy $\pi\in \Pi_{S}$, we let $\pi(x)$ denote the probability distribution over $U$ using which action is chosen when state is equal to $x$. Consider the set of actions $\{u: \pi(x,u)>0\}$ chosen by $\pi$ with positive probability in state $x$. Next, we show that if the state-action pairs $\{(x,u): \pi(x,u)>0\}$ corresponding to a sub-optimal policy $\pi$ have been visited for a sufficiently large number of times, then its index $I_{k}(\pi)$ is lower than the index of any optimal policy.

We now show that the quantity ${\te_{k,\pi}}$~\eqref{eq:index_theta} for a policy $\pi$ does not deviate ``a lot" from the MLE $\hat{p}_k$.  
\begin{lemma} \label{lemma:pinkser} 
Let $\pi\in \Pi_S$ be feasible for the true MDP $p$, i.e., $\bar{c}(\pi,p)\le c_{ub}$. Then,
		\begin{align*}
		|\theta_{k,\pi}(x,y,u)-\hat{p}_k(x,y,u)|\le d_2(x,u;\tau_k), (x,y,u) \in X \times X \times U,
		\end{align*}
	where
		\begin{align}\label{eq:d2}
	d_2(x,u;t):= \sqrt{\frac{\alpha(\tau_k)}{2~n_k(x,u) }\left(1 + \frac{1}{2a |X| p_{\min}}\right) }, \forall (x,u)\in X\times U.
		\end{align}
	Suppose that the number of visits to each state-action pair from the set $\{(x,u): \pi(x,u)>0\}$, until the time $\tau_k$ is lower bounded as follows,
	\begin{align}\label{eq:n_bound}
		n(x,u;\tau_k)> \frac{ \alpha(\tau_k)}{c^2} \; \forall (x,u) \text{ s.t. } \pi(x,u)>0.
	\end{align}
	Then, the distance between $\theta_{k,\pi}$ and $p$ is bounded as follows:
	\begin{align}\label{ineq:2}
		|\theta_{k,\pi}(x,y,\pi(x))-p(x,y,\pi(x))| < c\left[ \sqrt{.5\left(1+\frac{1}{a}\right)} + \frac{1}{\sqrt{a}}\right] ~~\forall\;x,y \in X .
	\end{align}

		\end{lemma}
		\begin{proof}
		Recall that the index $I_k(\pi)$~\eqref{eq:index} is given as follows,
	    \ali{    
		I_k(\pi)=\max_{\te \in \Theta: \bar{c}(\pi,\te)\le c_{ub}  }\bigg\{\alpha(\tau_k) \bar{r}(\theta,\pi) -\sum_{(x,u)} n_k(x,u)KL\left( \hat{p}_{k}(x,u),\te(x,u)  \right) \bigg\}.\label{recall:def_idx}
	}

Since $\te_{k,\pi}$ maximizes the expression on the r.h.s., the index can be written as follows, 
\begin{align}\label{rewrite_idx}
I_k(\pi)&= \alpha(\tau_k) \bar{r}(\theta_{k,\pi},\pi) -\sum_{(x,u)} n_k(x,u)KL\left( \hat{p}_{k}(x,u),\theta_{k,\pi}(x,u)  \right).
\end{align}
Since $\pi$ is feasible for the true MDP, we obtain the following
\ali{
I_k(\pi) \ge \alpha(\tau_k) \bar{r}(p,\pi) -\sum_{(x,u)} n_k(x,u)KL\left( \hat{p}_{k}(x,u),p(x,u)  \right).\label{ineq:1}
} 
By using inverse Pinsker's inequality we get
\ali{
KL\left( \hat{p}_{k}(x,u),p(x,u)  \right) \le \frac{1}{2 p_{\min}} \big(\sum\limits_{y\in X} |p(x,y,u)-\hat{p}_{k}(x,y,u)|\big)^2.
}
On the set $\cG_1$~\eqref{eq:goodset}, we have $|\te(x,y,u)-\hat{p}(x,y,u;t) |  \le d_1(x,u;t)$. Substituting this bound into the above inequality, we obtain the following
\ali{
KL\left( \hat{p}_{k}(x,u),p(x,u)  \right) \le  |X| \log\left( t^b |X|^{2} |U|\right).
}
Upon combining this with~\eqref{ineq:1}, we deduce that on $\cG_1$, the following holds,
\ali{
	I_k(\pi) \ge \alpha(\tau_k) \bar{r}(p,\pi) - |X| \log\left( t^b |X|^{2} |U|\right).
} 
Substituting the expression~\eqref{rewrite_idx} for $I_k(\pi)$ we get
\nal{
 \alpha(\tau_k) \bar{r}(\theta_{k,\pi},\pi) -\sum_{(x,u)} n_k(x,u)KL\left( \hat{p}_{k}(x,u),\theta_{k,\pi}(x,u)  \right) \ge \alpha(\tau_k) \bar{r}(p,\pi) - \frac{1}{2 p_{\min}}|X| \log\left( t^b |X|^{2} |U|\right),
}
or
\nal{
 \sum_{(x,u)} n_k(x,u)KL\left( \hat{p}_{k}(x,u),\theta_{k,\pi}(x,u)  \right) \le 	\alpha(\tau_k) \bar{r}(\theta_{k,\pi},\pi) - \alpha(\tau_k) \bar{r}(p,\pi) + \frac{1}{2 p_{\min}}|X| \log\left( t^b |X|^{2} |U|\right).
}
Since the average reward $\bar{r}(\theta,\pi) \in [0,1]$ for all $\theta \in \Theta$ and $\pi \in \Pi_{S}$, the above inequality reduces to
		\begin{align*}
		\sum_{(x,u)}	n_k(x,u)KL\left( \hat{p}_{k}(x,u),\theta_{k,\pi}(x,u)  \right) \le   \alpha(\tau_k) +\frac{1}{2 p_{\min}} |X| \log\left( t^b |X|^{2} |U|\right). 
		\end{align*}
		By using Pinsker's inequality~\citep{cover1999elements}, we can bound KL-divergence as follows:
		\begin{align*}
			|\theta_{k,\pi}(x,y,u)-\hat{p}_k(x,y,u)|^2 \le   \frac{1}{2}KL\left( \hat{p}_{k}(x,u),\te_{k,\pi}(x,u)  \right) \; \forall x,y \in X \text{ and } u \in U.
		\end{align*}
Upon combining the above two inequalities we deduce that on $\cG_1$ we have the following bound
\ali{
|\theta_{k,\pi}(x,y,u)-\hat{p}_k(x,y,u)|^2 \le \frac{\alpha(\tau_k) +\frac{|X|}{2 p_{\min}} \log\left( t^b |X|^{2} |U|\right)}{2~n_k(x,u) }.
}
Upon using $\alpha(t) =a\log\left( t^b |X|^{2} |U|\right)$~\eqref{eq:alpha} in the above, we get
\ali{
	|\theta_{k,\pi}(x,y,u)-\hat{p}_k(x,y,u)| \le \sqrt{\frac{\alpha(\tau_k)\left(1 + \frac{1}{2a |X| p_{\min}}\right)}{2~n_k(x,u) }}.
}
We now show~\eqref{ineq:2}. As is shown in Lemma \ref{lemma:confidence}, the distance between $p(x,y,u)$ and $\hat{p}_k(x,y,u)$ can be bounded by $d_1(x,u;\tau_k)$. We showed just now that the distance between $\theta_k(x,y,u)$ and $\hat{p}_k(x,y,u)$ can be bounded by $d_2(x,u;\tau_k)$. The proof then follows from the triangle inequality.
\end{proof}

\begin{corollary}\label{corollary:10}
Define the function
\ali{
	g(a) := \left[ \sqrt{.5\left(1+\frac{1}{a}\right)} + \frac{1}{\sqrt{a}}\right],
}
and let 
	\ali{
		c = \frac{\beta \Delta_{\min}}{g(a) \kappa_p |X|^2},
	}

where the parameter $\beta$ satisfies $\beta<1$. Let the number of visits to different state-action pairs satisfy the condition~\eqref{eq:n_bound}. Then,
{
	\ali{
		\bar{r}(\pi,\te_{k,\pi}) \le \bar{r}(\pi,p) + \beta \Delta_{\min}, \forall \pi \in \Pi_{S}.\label{ineq:bound_index}
	}
}	
Thus, the index $I_k(\pi)$ of a sub-optimal policy can be upper-bounded as follows,
\ali{
I_k(\pi) \le \alpha(T) \left[ \bar{r}(\pi,p) + \beta \Delta_{\min} \right].\label{upper_bound_idx}
}
Furthermore, if 
	\ali{	
		a > \frac{|X|^{2}|U|}{2(1-\beta) \Delta_{\min}p_{\min}}, \label{ineq:a_cond}
	}

then 
\ali{
I_k(\pi) \le I_k(\pi\ust), \mbox{ where } \pi\ust \in \Pi\ust(p).\label{ineq:index}
}
$\blacksquare$
\end{corollary}
\begin{lemma}\label{lemma:bound_2}
Let $\pi$ be an infeasible policy that satisfies the following, 
\ali{
\bar{c}(\pi,p)>c_{ub}+ (w+c) \kappa_p,\label{def:cost_margin}
} 
where $w>0$ satisfies~\eqref{cond:w}. Assume that the number of visits to different state-action pairs satisfy the condition~\eqref{eq:n_bound}. Then, on the set $\cG_1$, the index $I_k(\pi)$ can be upper-bounded as follows:
\ali{
I_k(\pi) \le \alpha(T) - \alpha(T) \frac{w^2}{c^2}.
}
If 
\ali{
		w \ge c \left[   1 + \frac{|X|^{2}|U|}{2a p_{\min}}  \right]^{.5},\label{cond:w}
}
then the index $I_k(\pi)$ is less than the index of an optimal policy.
\end{lemma}

\begin{proof}
	Throughout this proof we assume that we restrict to a sample path lying in $\cG_1$. Let $\te\in \Theta$ be such that $\bar{c}(\pi,\te)\le c_{ub}$. Since the cost incurred by $\pi$ on the system with parameter value equal to true value $p$ satisfies $\bar{c}(\pi,p)\ge c_{ub} +  (w+c)\kappa_p$, it follows from the perturbation bound in~Lemma \ref{lemma:cho_meyer} that $\|\te-p \|\ge w+c$. Thus, we have 
\ali{
\|\te - \hat{p}_k\| &= \| (\te -  p)  + (p - \hat{p}_k)\| \notag \\
& \ge  \| (\te -  p)\| - \|(p - \hat{p}_k)\| \notag \\
&\ge w+c - c \notag \\
&=w (say),\label{ineq:distance}
}
where the last inequality also uses the fact that on $\cG_1$ we have $\|p-\hat{p}_k\|\le d_1(x,u;t)$. Now, consider the problem~\eqref{eq:rbmle_2} that is to be solved for deriving $\te_k$. The expression within the braces in~\eqref{eq:rbmle_2} that corresponds to maximization over the set of policies $\Pi_S$ for a fixed $\te$ can be bounded as follows
\ali{
& \alpha(T) R\ust(\te) - \sum_{(x,u)} n_k(x,u) KL(\hat{p}_k(x,u), \te(x,u) ) \notag\\
& \le \alpha(T) R\ust(\te) -  \frac{\alpha(T)}{c^2} (\hat{p}_k - \te)^2 \notag\\
&\le \alpha(T) R\ust(\te) -  \frac{\alpha(T)}{c^2} w^2,\label{ineq:6}
}
where the first inequality follows from , while the second from~\eqref{ineq:distance}. It follows from Lemma~\ref{lemma:suff_prec_optimal} that the expression~\eqref{ineq:6} is less than the index of an optimal policy if the following holds
\nal{
	\alpha(\tau_k)R\ust(p)\left(1-\frac{|X|^{2}|U|}{2a p_{\min} R \ust(p)}\right) > \alpha(T) R\ust(\te) -  \frac{\alpha(T)}{c^2} w^2,
}
i.e.,
\nal{
	R\ust(p)\left(1-\frac{|X|^{2}|U|}{2a p_{\min} R \ust(p)}\right) > R\ust(\te) -  \frac{w^2}{c^2}
}
i.e.,
\nal{
	\frac{w^2}{c^2} \ge R\ust(\te)-R\ust(p) + \frac{|X|^{2}|U|}{2a p_{\min}}
}
Since $R\ust(\te)\le 1$, the above is clearly satisfied if we have
\ali{	
		w \ge c \left[   1 + \frac{|X|^{2}|U|}{2a p_{\min}}  \right]^{.5}.	
}
\end{proof}

\begin{lemma}\label{lemma:g2}
Define ${\cG}_2$ to be the following set 		
\begin{align}\label{eq:g2}
\cG_2 :=\left\{	\omega: n(x,u;T) \ge  \frac{y_{x,u}}{2}  - \sqrt{y_{x,u}\log T} \text{ for all } (x,u) \right\},
\end{align}
where $y_{x,u} := \sum_{k \in \cK_{x,u} } \Bigl\lfloor \frac{  |\cE_k | }{2 T_p} \Bigr\rfloor$  ,
and $\cK_{x,u}$ denotes the set of indices of those episodes up to time $T$ in which action $u$ is taken when state is equal to $x$. 	
Then,
$$
\bP\left({\cG}_2 \right) \geq 1- \frac{|X||U|}{T}.
$$ 
\end{lemma}
Follows from (Lemma 11, \cite{techreport}).

		\section{Regret Analysis}
		We now utilize the results of previous section in order to obtain an upper bound on the expected regret of the RBMLE algorithm. We begin by decomposing the cumulative regrets $\cR_{r}(\phi,T), \cR_{c}(\phi,T)$ of the learning rule $\phi$, into the sum of episodic regrets as follows,
		\ali{
		\cR_{r}(\phi,T) &= \sum_{k} \Big( R\ust(p) |\cE_k| - \sum_{t\in\cE_k} r(x(t),u(t)) \Big),\\
		\cR_{c}(\phi,T) &= \sum_{k} \Big( \sum_{t\in\cE_k} c(x(t),u(t)) - c_{ub} |\cE_k|\Big).	
	}
		Since the RBMLE algorithm implements a single stationary policy $\pi_k$ during $\cE_k$, we obtain the following bounds on the expected regrets (Lemma 12, \cite{techreport}),
		\begin{align}
		\bE  \cR_{r}(\phi,p,T) &\le \sum_{k=1}^{K(T)} \bE \Big( R\ust(p) |\cE_k| - |\cE_k| \bar{r}(\pi_k,p) \Big)  + T_pK(T),\label{def:regret_dec_1}\\
		\bE  \cR_{c}(\phi,p,T) &\le \sum_{k=1}^{K(T)} \bE \Big(  |\cE_k| \bar{c}(\pi_k,p) - \tilde{c}_{ub} |\cE_k|\Big)  + T_pK(T),		\label{def:regret_dec_2}
		\end{align}
		where $K(T)$ is the number of episodes till $T$. The first summation can be regarded as the sum of the regrets arising from the policies chosen
		in the episodes $k=1,2, \ldots, K(T)$, assuming that each episode is started with a steady-state distribution for the state corresponding to the policy
		chosen in that episode. The last term $T_p K(T)$ arises because the system does not start in a steady state in each episode.
		
		We now state the main result of this paper. It shows that the expected regrets $\cR_{r}(T),\cR_{c}(T)$ of the RBMLE algorithm are bounded by $c' \log T + c''$ for all $T$:
		\begin{theorem}\label{th:regret}
The reward and cost regrets of RBMLE based-policy can be upper-bounded as follows
			\begin{align*}
			\bE  \cR_{r}(T)  &\le 2 |X||U|T_p(c_1n_c+K(T)) +  |X||U| + \frac{8}{|X|^2|U|} + \left(\log_2 T\right) T_p~\text{   for all } T,\\
			\bE  \cR_{c}(T) &\le 2 |X||U|T_p(c_1n_c+K(T)) +  |X||U| + \frac{8}{|X|^2|U|} + \left(\log_2 T\right) T_p~\text{   for all } T,			
			\end{align*}
		\end{theorem}
	\begin{proof}
		The decomposition~\eqref{def:regret_dec_1}-\eqref{def:regret_dec_2} shows that the episodic regrets during $\cE_k$ due to using $\pi_k$ have the following properties: \\
		(a) Episodic regrets for cost and reward can be bounded by $0$ in those episodes in which $\pi_k$ is optimal, i.e., $\pi_k\in \Pi\ust(p)$.\\
		(b) if $\pi_k$ is feasible, but not optimal, then the episodic reward regret is bounded by the length of the episode $|\cE_k |$ (since $r(x,u)\le 1$), while the cost regret is bounded by $0$. \\
		(c) If $\pi_k$ is infeasible, then its cost satisfies $\bar{c}(\pi,p) > \tilde{c}_{ub}$. In this case, the episodic cost regret is bounded by $1$ (since $c(x,u)<1$). Moreover, the reward regret is also bounded by $1$.

Define the ``good set" $\cG := \cG_1 \cap {\cG}_2 $, where $\cG_1$ is as in~\eqref{eq:goodset} and $\cG_2$ is as in Lemma~\ref{lemma:g2}.

We first analyze the regrets on $\cG$. This regret can be further decomposed into the following three types, each of which are bounded separately:\\

\textit{(i) Regret due to suboptimal episodes on $\cG$}: Within such an episode $\cE_k$, a sub-optimal policy is played, i..e, $\pi_k$ is sub-optimal. Note that a policy fron the set $\Pi_F$ is sub-optimal because either it is infeasible, or because its reward is less than $r\ust(p)$. On $\cG$, confidence intervals $\cC(t)$, defined in~\eqref{def:d1}, hold true for all episode starting times $\tau_k$, $k \in [1,K(T)]$, and also (\ref{eq:g2}) holds. 
 Hence, it follows from Corollary~\ref{corollary:10} that any sub-optimal policy with reward less than $r\ust(p)$ will not be played if $n(x,u;\tau_k) >  n_c$ for all $(x,u)$. A similar conclusion holds for any infeasible policy.  In summary, if $n(x,u;\tau_k) >  n_c$ for all $(x,u)$ then the episodic regrets can be taken to be $0$. 

We now upper bound the number of time-steps in such ``sub-optimal'' episodes $\cK_{x,u}$ in which control $u$ is applied in state $x$.  If $n(x,u;\tau_k) >  n_c$ for all $(x,u)$ does not hold, then there exists at least one state, action pair $(x,u)$ with $n(x,u;\tau_k) \le n_c$. Since $n(x,u;T)\le n_c$, we have $n_c \ge  \frac{y_{x,u}}{2}  - \sqrt{y_{x,u}\log T}$. Note that $n_c \ge\kappa_p^2\log T$.  Then there exists $c_1<\frac{11}{\kappa^2_p} $ such that $y_{x,u} \leq c_1n_c$ (Lemma 13, \cite{techreport}).
		So
		$
			\sum_{k \in \cK_{x,u} }  |\cE_k |  \le 2 T_pc_1n_c+2|\cK_{x,u}|T_p
		$.
		Note that $\cK_{x,u}<K(T)$, where $K(T)$ is the total number of episodes till $T$. We let $\cR_{r,1}$ ($\cR_{c,1}$) be the total reward regret until $T$ due to suboptimal episodes on the good set $\cG$. Then 
		\begin{align}\label{ineq:r1}
			\cR_{r,1} &\le \sum_{(x,u)}\sum_{k \in \cK_{x,u} } |\cE_k |\le 2 |X||U|T_p(c_1n_c+K(T)),\\
			\cR_{c,1} &\le \sum_{(x,u)}\sum_{k \in \cK_{x,u} } |\cE_k |\le 2 |X||U|T_p(c_1n_c+K(T)).
		\end{align}
	\\
\textit{(ii) Regret on ${\cG}_2^c$.} Denote these expected regrets by $\cR_{r,2}$ and $\cR_{c,2}$. The probability of the set where the conclusions of (\ref{eq:g2}) do not hold true for a state-action pair $(x,u)$ is upper bounded by $\frac{|X||U|}{T}$. Since the sample-path regret can be trivially upper bounded by $T$, it follows that 
$$
\cR_{r,2},\cR_{c,2}  \le |X||U|.
$$ 
\\\\	
	\textit{(iii) Regret on $\cG_1^c$}. It follows from Lemma~\ref{lemma:confidence} that for any episode $k \in [0,K(T)]$, the probability of failure of the confidence interval $C(\tau_k)$ can be upper bounded by $\frac{2}{|X|^2|U|\tau_k^{2b-1}}$. The expected regrets (reward and cost) in each such episode can be bounded by the length of the episode $\cE_k$. We let $\cR_{r,3}$ ($\cR_{c,3}$) be the total expected reward (cost) regret until $T$ due to the failure of confidence intervals. It can be upper-bounded as:			
\begin{align*}
				\mathcal{R}_{r,3}& \leq \sum_{k=1}^{K(T)} \frac{2(\tau_{k+1}-\tau_{k})}{|X|^2|U|\tau_k^{2b-1}}\le \sum_{k=1}^{\infty} \frac{2}{|X|^2|U|\tau_k^{2b-2}}\left(\frac{\tau_{k+1}}{\tau_k}\right)		\le \sum_{k=1}^\infty \frac{4}{|X|^2|U|\tau_k^{2b-2}} \le \frac{8}{|X|^2|U|},\\
				\mathcal{R}_{c,3}& \leq \sum_{k=1}^{K(T)} \frac{2(\tau_{k+1}-\tau_{k})}{|X|^2|U|\tau_k^{2b-1}}\le \sum_{k=1}^{\infty} \frac{2}{|X|^2|U|\tau_k^{2b-2}}\left(\frac{\tau_{k+1}}{\tau_k}\right)		\le \sum_{k=1}^\infty \frac{4}{|X|^2|U|\tau_k^{2b-2}} \le \frac{8}{|X|^2|U|}.
\end{align*}
\textit{(iv) Additional regret due to not starting in a steady state in each episode}.
\\The RBMLE algorithm implements a stationary policy $\pi_k$ during $\cE_k$. The total expected reward within an episode depends upon the starting state. This means that there is an additional loss if it starts in an unfavorable state. This can be upper-bounded by $T_p$. Letting $\mathcal{R}_{r,4},\cR_{c,4}$ be the total expected regrets due to not starting in a steady state in each episode, we infer that $\cR_{r,4},\cR_{c,4} \le K(T)T_p$ where, $K(T)=\lceil\log_2 T \rceil$ is the total number of episodes till $T$.\\\\
The upper bound on regrets is derived by by adding the bounds on $\cR_{r,1},\cR_{r,2},\cR_{r,3},\cR_{r,4}$ ( or $\cR_{c,1},\cR_{c,2},\cR_{c,3},\cR_{c,4}$ for cost regret).
\end{proof}
\section{Computational Issues}\label{sec:solve_cmdp}
We will see how to efficiently solve the problems~\eqref{eq:rbmle_2} and~\eqref{def:pi_k} so as to obtain $\te_k$ and $\pi_k$. We merge these into a single optimization problem, and repeat it below for convenience of readers:
	\begin{align}
	\max_{\te \in \Theta} \bigg\{ \max_{\pi \in \Pi_{S}}\bigg\{ \alpha(\tau_k) \bar{r}(\theta,\pi) -\sum_{(x,u)} n_k(x,u)KL\left( \hat{p}_{k}(x,u),\te(x,u)  \right): \bar{c}(\te,\pi) \le c_{ub} \bigg\} \bigg\}.\label{rbmle_prob}
\end{align}
We begin by embedding it into a certain ``extended CMDP".

\textbf{Extended CMDP}: The problem~\eqref{rbmle_prob} can be viewed as an \textit{extended CMDP} in which the optimizer\slash controller has to choose the following two quantities at each time $t$: \\
(i) controls $u(t)\in U$, \\
(ii) MDP parameters $\te \in \Theta$.\\
We can merge these two controls, and let $u^{+}(t):= (u(t),\te(t))\in U \times \Theta$ be the controls for this \textit{extended CMDP}. We denote the controlled transition probabilities for this extended MDP by $P^{+}$. These are given as follows:
\ali{
	P^{+}(x,u^{+},y) =  \te(x,u,y),\label{def:extend_ptm}
}
where 
\nal{
u^+ = (u,\te)
}
is control for the extended CMDP. The reward function and cost function for the extended CMDP are the same as for the original CMDP, i.e.
$$
r(x,u^{+}) := r(x,u), c(x,u^+) := c(x,u).
$$
Let $\pi^{+}$ be a policy for this extended CMDP that chooses controls $u^{+}(t)$. Problem~\eqref{rbmle_prob} requires us to solve the following extended CMDP:
\ali{
\max_{\pi^{+}} &\liminf_{T\to\infty}  \frac{1}{T} \bE_{\pi^{+}}\left( \sum_{t=1}^{T} r(x(t),u^{+}(t))\right),\label{ext_cmdp_1}\\
\mbox{ s.t. } &\limsup_{T\to\infty}  \frac{1}{T} \bE_{\pi^{+}}\left( \sum_{t=1}^{T} c(x(t),u^{+}(t))\right)\le c_{ub}.\label{ext_cmdp_2}
}
We let $\cP(U\times \Theta)$ be the set of probabliity measures on the set $U \times \Theta$.  
\begin{definition}[Stationary Randomized Policy for Extended CMDP]
These policies sample the control action $u^{+}(t)$ by using a probability measure from the class $\cP(U\times \Theta)$ that is a function of the current state value $x(t)$. Such a stationary randomized policies policy can be parameterized by measures $\left\{\pi(x): x\in X\right\}$ such that $\pi(x)\in \cP(U\times \Theta)$. Denote the set of these policies by $\Pi^{+}_{S}$.
\end{definition}
Te next result shows that there exists a stationary policy that is optimal for~\eqref{ext_cmdp_1}-\eqref{ext_cmdp_2}. 
\begin{lemma}\label{lemma:stat_opt}
	For the extended CMDP~\eqref{ext_cmdp_1}-\eqref{ext_cmdp_2}, there exists an optimal controller within the class of stationary randomized policies $\Pi^{+}_S$ for the extended CMDP. 
\end{lemma}
Note that a measure from the set $\cP(U\times \Theta)$ is an infinite dimensional object. Hence, optimizing amongst the class of policies $\Pi^{+}_{S}$ in order to solve~\eqref{ext_cmdp_1}-\eqref{ext_cmdp_2} requires us to solve an infinite dimensional LP~\cite{hernandez2012further}. Since this optimization problem would be computationally infeasible, we next propose to approximate it by an optimization problem that involves finitely many decision variables.

We begin by constructing a finite set $\Theta_{\epsilon}$ that is an ``$\epsilon$-approximation" to the set $\Theta$. It has the following property: for each $\te \in \Theta$, there exists at least one point $\sigma_{\te}\in \Theta_{\epsilon}$ such that $\|\te-\sigma_{\te}\|\le \epsilon$. We let $\Pi^{+}_{S,\epsilon}$ denote the set of those stationary policies for the extended CMDP~\eqref{ext_cmdp_1}-\eqref{ext_cmdp_2} that restrict to the set $\Theta_{\epsilon}$ while choosing $\te$ (instead of choosing from the larger set $\Theta$). The following result shows how to approximately solve the original extended CMDP~\eqref{ext_cmdp_1}-\eqref{ext_cmdp_2} by restricting to the class $\Pi^{+}_{S,\epsilon}$. It also derives an upper-bound on the corresponding  ``truncation error".
\begin{theorem}
Consider the following CMDP
\ali{
	\max_{\pi^{+}\in \Pi^{+}_{S,\eps}  } &\liminf_{T\to\infty}  \frac{1}{T} \bE_{\pi^{+}}\left( \sum_{t=1}^{T} r(x(t),u^{+}(t))\right),\label{ext_cmdp_1_eps}\\
	\mbox{ s.t. } &\limsup_{T\to\infty}  \frac{1}{T} \bE_{\pi^{+}}\left( \sum_{t=1}^{T} c(x(t),u^{+}(t))\right)\le c_{ub},\label{ext_cmdp_2_eps}
}
that can be considered as an approximation to the ``true CMDP"~\eqref{ext_cmdp_1}-\eqref{ext_cmdp_2}. The optimal value of~\eqref{ext_cmdp_1_eps}-\eqref{ext_cmdp_2_eps} is within $\cdots$ of the optimal value of~\eqref{ext_cmdp_1}-\eqref{ext_cmdp_2}.
\end{theorem}
\begin{proof}
	Consider a CMDP that is obtained by modifying the CMDP~\eqref{ext_cmdp_1}-\eqref{ext_cmdp_2} as follows: the upper bound on cost expenditure is replaced with $c_{ub}-\epsilon$. It follows from sensitivity analysis of CMDPs Theorem~\ref{th:perturbation_cmdp} that its optimal reward is greater than 
	$$
	r_1 := r\ust -  \eps \frac{\hat{\eta}}{\eta},
	$$ 
	(where $r\ust$ is optimal value of \eqref{ext_cmdp_1}-\eqref{ext_cmdp_2}). Let $\pi_1$ solve this modified CMDP. $\pi_1$ chooses action according to the measure $\pi\ust(x)\in \cP(U \times \Theta)$ when state is $x$. We approximate $\pi_1$ as follows, so that now it places probability mass on only the discrete set $U \times \Theta_{\eps}$. Whenever $\pi_1$ suggests an action $(u,\te)$, then the discretized policy maps it to $(u,\sigma_{\te})$, i.e., the nearest point in the set $U \times \Theta_{\eps}$. Since we have that $\|\sigma_{\te} -\te\|\le \epsilon$, it follows from Lemma~\ref{lemma:cho_meyer} that the average reward of this modified policy (say $\pi_{1,\eps}$) is greater than 
	$$
	r_1-  \kappa_{\pi\ust} \epsilon = r\ust -  \eps \frac{\hat{\eta}}{\eta}-\kappa_{\pi\ust} \epsilon,
	$$
	and average cost is lesser than $c_{ub}-\epsilon + \kappa_{\pi\ust} \epsilon\le c_{ub}$. Since $\pi_{1,\eps}\in \Pi^{+}_{S,\eps}$, we deduce that the optimal value of~\eqref{ext_cmdp_1_eps}-\eqref{ext_cmdp_2_eps} is greater than or equal to $r\ust -  \eps \frac{\hat{\eta}}{\eta}-\kappa_{\pi\ust} \epsilon$.
	
\end{proof}
We now discuss how to obtain $\pi_k$, the policy that is to implemented within the $k$-th episode. Consider the following linear program\footnote{It is supposed to yield a solution to the CMDP~\eqref{ext_cmdp_1_eps}-\eqref{ext_cmdp_2_eps}.} in which the decision variables are given by $\{ \mu(x,u,\te): x \in X, u \in U,\te \in \Theta_{\eps} \}$:
\begin{align}
	&\max \sum_{x \in \cX} \sum_{u \in \cU} \sum_{\te \in  \Theta_{\eps}} \mu(x,u,\te)r(x,u),\label{overallobj_appx}
	&\intertext{subject to }\notag\\
	& \sum_{(u,\te)} \mu(x,u,\te) = \sum_{(y,v,\te)} \mu(y,v,\te) P^{+}(y, (v,\te),x), \label{const1n_appx} \\
	& \mu(x,u,\te) \geq 0, \\
	&\sum_{x \in \cX}\sum_{u} \sum_{\te \in  \Theta_{\eps}}\mu(x,u,\te) = 1,  \label{const2n_appx} \\
	& \sum_{(x,u,\te)} \mu(x,u,\te) c(x,u) \leq U_B . \label{constall_appx}
\end{align} 
Let $\mu\ust$ be a solution. Define $\pi_k$ to be the stationary randomized policy that randomizes amongst the actions using the following distribution when state is equal to $x$: 
\begin{align}
	\frac{  \sum\limits_{\te \in  \widetilde{\Theta}_{\epsilon}}    \mu\ust(x,u,\te)}{\sum_{(v,\te)} \mu\ust(x,v,\te)}.
\end{align}
(In case the denominator above is zero, choose any $\pi\ust_n(x_n,u_n,\te_n) \geq 0$ satisfying
$\sum_{v_n,\te_n} \pi\ust_n(x_n,v_n,\te_n) = 1$).

\begin{assumption}\label{assum:strict}
	The CMDP is strictly feasible, i.e.	
	there exists a policy $\pi_{feas.} \in \Pi^{+}$ and a constant $\epsilon>0$ such that the average cost under $\pi_{feas.}$ is strictly below $c_{ub} - \epsilon$. We denote
	\ali{
		\eta := c_{ub} - \epsilon - \bar{c}(\pi_{feas.}). \label{def:eta}
	}
\end{assumption}

\section{Proof of Lemma \ref{lemma:confidence}}\label{proof:confidence}
Consider the scenario where the number of visits to $(x,u)$ is fixed at $n_{x,u}$, and let $\hat{p}(x,y,u)$ be the resulting estimates. Consider the event $\left\{|p(x,y,u)-\hat{p}(x,y,u) | > r\right\}$, where $x,y \in X, u \in U$ and $r>0$. It follows from the Azuma-Hoeffding's inequality \citep{mitzenmacher2017probability} that the probability of this event is upper bounded by $2\exp(-2n_{x,u}r^2)$. Therefore, 
\begin{align*}
	\bP\left(|p(x,y,u)-\hat{p}(x,y,u) | > \sqrt{\frac{ \log\left(t^b |X|^{2} |U|\right) }{n_{x,u}}   }     \right) \leq 2\left(\frac{1}{t^b|X|^{2}|U|}\right)^2.
\end{align*} 
Utilizing union bound on the number of plays of action $u$ in state $x$ until time $t$ and considering all possible state-action-state pairs, we get
\begin{align*}
	\bP(p \notin \cC(t))\le \frac{2}{|X|^2|U|t^{2b-1}}\;\forall\; t \in [1,T] .
\end{align*}
\section{Proof of Lemma \ref{lemma:suff_prec_optimal}}\label{proof:suff_prec_optimal}
\begin{proof} 
	Consider an optimal policy $\pi\ust\in \Pi\ust(p)$.	Recall that its RBMLE index (\ref{eq:index}) is given as follows
	\begin{align}   
		I_k(\pi\ust)=\max_{\te \in \Theta: \bar{c}(\pi\ust,\te)\le c_{ub}  }\bigg\{\alpha(\tau_k) \bar{r}(\theta,\pi\ust) -\sum_{(x,u)} n_k(x,u)KL\left( \hat{p}_{k}(x,u),\te(x,u)  \right) \bigg\}.
	\end{align} 
	Since the original CMDP is feasible, i.e. $\bar{c}(\pi\ust,p)\le c_{ub} $, we evaluate the expression within the braces in the above for $\te = p$ to get (let $R\ust(p)$ be the optimal value of the true CMDP)
	\begin{align*}
		I_k(\pi \ust) 	&\ge \bigg\{  \alpha(\tau_k)\bar{r}(\pi\ust,p) -\sum_{(x,u)} n_k(x,u)KL\left( \hat{p}_{k}(x,u),p(x,u)  \right) \bigg\}\\ 
		&= \bigg\{  \alpha(\tau_k)R\ust(p) -\sum_{(x,u)} n_k(x,u)KL\left( \hat{p}_{k}(x,u),p(x,u)  \right) \bigg\}\\
		&\ge 
		\bigg\{  \alpha(\tau_k)R\ust(p) -\sum_{(x,u)} n_k(x,u)\frac{\big(\sum\limits_{y\in X} |p(x,y,u)-\hat{p}_{k}(x,y,u)|\big)^2}{2p_{\min}} \bigg\},
	\end{align*}
	where the first inequality follows since $\te_{k,\pi\ust}$~\eqref{eq:index_theta}  maximizes the objective in~\eqref{eq:index}, 
	while the second inequality follows from the inverse Pinkser's inequality \citep{cover1999elements} and Assumption~\ref{assum:1}.
	Since on $\cG_1$, we have that $|p(x,y,u)-\hat{p}_{k}(x,y,u)|<d_1(x,u;t)$  for all $(x,y,u) \in X \times X \times U$, upon substituting this into the above inequality we obtain the following
	\begin{align*}
		I_k(\pi \ust)\ge 	\bigg\{  \alpha(\tau_k)R\ust(p) - \frac{|X|^2|U|}{2p_{\min}}\log\left(t^b |X|^{2} |U|\right) \bigg\}
		= \alpha(\tau_k)R\ust(p)\left(1-\frac{|X|^{2}|U|}{2a p_{\min} R \ust(p)}\right).
	\end{align*}
\end{proof}

\section{Proof of Corollary~\ref{corollary:10}}\label{proof:coro_10}
\begin{proof}
	The bound~\eqref{ineq:bound_index} follows from the bound~\eqref{ineq:2} derived in Lemma~\ref{lemma:pinkser}, and Lemma~\ref{lemma:cho_meyer}.~To prove~\eqref{ineq:index}, we compare the upper-bound~\eqref{upper_bound_idx} on the index of $\pi$, with the lower-bound on index of an optimal policy that was derived in Lemma~\ref{lemma:suff_prec_optimal}. We see that $I_k(\pi)$ is less than $I_k(\pi\ust)$ if the following holds true,
	\nal{
		\alpha(\tau_k)R\ust(p)\left(1-\frac{|X|^{2}|U|}{2a p_{\min} R \ust(p)}\right) > \alpha(T) \left[ \bar{r}(\pi,p) + \beta \Delta_{\min} \right]
	}
	or
	\ali{
		R\ust(p) -\bar{r}(\pi,p) -\frac{|X|^{2}|U|}{2a p_{\min}}> \beta \Delta_{\min}.
	}
	Since $R\ust(p) -\bar{r}(\pi,p) \ge \Delta_{\min}$, the above condition is satisfied if $a$ satisfies~\eqref{ineq:a_cond}.
\end{proof}

\section{Auxiliary Results}
The following results are from \citep{cho_00} and \citep{auer_07} respectively.
\begin{lemma}\citep{cho_00}\label{lemma:cho_meyer} 
	Consider a stationary policy $\pi$ and $\te$ be an MDP parameter that satisfies
	\begin{align}\label{def:suff_visits}
		|\te(x,y,\pi(x))-p(x,y,\pi(x))|< \frac{\epsilon}{\kappa_p |X|^{2}}, \; \forall \; x,y \in X,
	\end{align}
	where $\epsilon>0$ and $\kappa_p$ is the conductivity. We then have that 
	\begin{align*}
		| \bar{r}(\te,\pi) - \bar{r}(p,\pi) | < \epsilon.
	\end{align*}
Similarly for costs.
\end{lemma}
\begin{lemma}\citep{auer_07}\label{lemma:azum_visits}
	Let $\cK_{x,u}$ denote the indices of those episodes up to time $T$ in which action $u$ is taken when state is equal to $x$. Then
	\begin{align}\label{eq:azum_visits}
		\bP\left(  n(x,u;T) \ge  \frac{y_{x,u}}{2}  - \sqrt{y_{x,u}\log T} ~ \forall~{x,u} \right) \ge 1 - \frac{|X||U|}{T},
	\end{align}
	for all state-action pairs $(x,u)$, where
	\begin{align*}
		y_{x,u} := \sum_{k \in \cK_{x,u} } \Bigl\lfloor \frac{  |\cE_k| }{2 T_p} \Bigr\rfloor .
	\end{align*}
\end{lemma}
\begin{lemma}(Lemma 2, \cite{auer_07})\label{lemma:mix}
	Let $\pi$ be a stationary policy. Consider a controlled Markov process that starts in state $x$ and evolves under $\pi$. We then have that 
	\begin{align*}
		\bE_{x}\left( \sum_{t=1}^{T} r(x(t),u(t)) \right)  \ge TJ(\pi,p) - T_p.
	\end{align*}
\end{lemma}
\begin{lemma}\label{lemma:bound_on_c}
	Consider the following function $f(x)$ such that $a_0>a_1>0$,
	\begin{align}
		f(x)=x-2\sqrt{a_1x}-2a_0.
	\end{align}
Then there exist $x_0<11a_0$ such that $f(x)>0$ for all $x>x_0$.
\end{lemma}
\begin{proof}Note that $f(a_1)<0$ and 
	\begin{align*}
\frac{\partial f}{\partial x}=1-\sqrt{\frac{a_1}{x}}>0 ~\forall~ x>a_1.
\end{align*}
The result follows since $f(11a_0)=9a_0-2\sqrt{11a_0a_1}>(9-2\sqrt{11})a_0>0.$
\end{proof}
\section{Perturbation Analysis of CMDPs}
We derive some results on the variations in the value of optimal reward of the CMDP 
\begin{align}
	\max_{\pi} &\liminf_{T\to\infty }\frac{1}{T} \mathbb{E}_{\pi}\sum_{t=1}^{T} r( x(t),u(t)) \label{eq:cmdp_modf_obj}\\
	\mbox{ s.t. }&  \limsup_{T\to\infty }\frac{1}{T} \mathbb{E}_{\pi}\sum_{t=1}^{T} c(x(t),u(t)) \leq c_{ub}.\label{eq:cmdp_modf_constr}
\end{align}
as a function of the cost budget $c_{ub}$. Consider a cost budget $\hat{c}_{ub}$ that satisfies
\begin{align}\label{ineq:c_range}
	c_{ub}-\epsilon \le \hat{c}_{ub} \le c_{ub}
\end{align}
where $\epsilon>0$. 
\begin{lemma}\label{lemma:lagrange_ub} 
	Let the MDP $p$ be unichain and strictly feasible. Let $\lambda\ust$ be an optimal dual variable\slash Lagrange multiplier associated with the CMDP~\eqref{eq:cmdp_modf_obj}-\eqref{eq:cmdp_modf_constr}. Then, $\lambda\ust$ satisfies 
	\begin{align*}
		\lambda\ust \le \frac{\hat{\eta} }{\eta},
	\end{align*}
	where the constant $\eta$ is as in~\eqref{def:eta}, while $\hat{\eta}$ is defined as follows,
	\begin{align}\label{eq:def_eta}
		\hat{\eta}:= \max_{(x,u)\in\mathcal{S}\times\mathcal{A}}r(s,a) - \min_{(s,a)\in\mathcal{S}\times\mathcal{A}}r(s,a).  
	\end{align}
\end{lemma}
\proof Within this proof, we let $\pi^{\star}(\hat{c}^{ub})$ denote an optimal stationary policy for~\eqref{eq:cmdp_modf_obj}-\eqref{eq:cmdp_modf_constr}. Let $\pi_{feas.}$ be the feasible policy was that satisfies $\bar{c}(\pi_{feas.})\le c^{ub}-\eta$.
We have
\begin{align*}
	\max_{(s,a)\in\mathcal{S}\times\mathcal{A}}r(s,a) &\geq  \bar{r}(\pi^{\star}(\hat{c}^{ub}))\\
	&= \bar{r}(\pi^{\star}(\hat{c}^{ub})) +  \lambda\ust \left(\hat{c}^{ub}-\bar{c}( \pi^{\star}(\hat{c}^{ub})  \right) \\
	&\geq\bar{r}(\pi_{feas.}) + \lambda\ust \left(\hat{c}^{ub}-\bar{c}(\pi_{feas.}) \right) \\
	&\geq \min_{(s,a)\in \mathcal{S}\times\mathcal{A} }r(s,a) +  \lambda\ust \left(\hat{c}^{ub}-\bar{c}(\pi_{feas.}) \right) \\
	&\geq \min_{(s,a)\in \mathcal{S}\times\mathcal{A} }r(s,a) + \eta \sum_{i=1}^{M} \lambda\ust_i,
\end{align*}
where the second inequality follows since a policy that is optimal for the problem~\eqref{eq:cmdp_modf_obj}-\eqref{eq:cmdp_modf_constr} maximizes the Lagrangian $\bar{r}(\pi) +  \lambda \left(\hat{c}^{ub}-\bar{c}( \pi)  \right)$ when the Lagrange multiplier $\lambda$ is set equal to  $\lambda\ust$~\citep{bertsekas1997nonlinear}. Rearranging the above inequality yields the desired result. $\square$

\begin{theorem}\label{th:perturbation_cmdp}
	Let the MDP $p$ satisfy Assumption~\ref{assum:1} and Assumption~\ref{assum:strict}. If $r^{\star}(\hat{c}^{ub})$ denotes optimal reward value of the CMDP
	\begin{align}
		\max_{\pi} &\liminf_{T\to\infty }\frac{1}{T} \mathbb{E}_{\pi}\sum_{t=1}^{T} r( x(t),u(t)) \label{eq:cmdp_modf_obj_hat}\\
		\mbox{ s.t. }&  \limsup_{T\to\infty }\frac{1}{T} \mathbb{E}_{\pi}\sum_{t=1}^{T} c(x(t),u(t)) \leq c_{ub},\label{eq:cmdp_modf_constr_hat}
	\end{align}
 and $r\ust$ is optimal reward of problem~\eqref{eq:cmdp_modf_obj}-\eqref{eq:cmdp_modf_constr}, then we have that
	\begin{align*}
		r^{\star} - r^{\star} (\hat{c}_{ub}) \leq \left( c_{ub} - \hat{c}_{ub}  \right) \frac{\hat{\eta}}{\eta},
	\end{align*}
	where $\hat{\eta}$ is as in~\eqref{eq:def_eta}, $\eta$ is as in~\eqref{def:eta}, and $\hat{c}$ satisfies~\eqref{ineq:c_range}.
\end{theorem}
\proof 
Since under our assumption, both the CMDPs are feasible,
we can use the strong duality property of LPs in order to conclude that the optimal value of the primal and the dual problems for both the CMDPs are equal. Thus,
\begin{align}
	r^{\star} &= \sup_{\pi} \inf_{\lambda} ~~\bar{r}(\pi)+\sum_{i=1}^{M} \lambda_i\left(c^{ub}_i - \bar{c}_i(\pi)\right),\label{ineq:12}\\
	r^{\star} (\hat{c}^{ub}) &= \sup_{\pi} \inf_{\lambda}~~ \bar{r}(\pi)+\sum_{i=1}^{M} \lambda_i \left(\hat{c}_i^{ub} - \bar{c}_i(\pi)\right).\label{ineq:13}
\end{align}
Let $\pi^{(1)},\pi^{(2)}$ and $\lambda^{(1)},\lambda^{(2)}$ denote optimal policies and optimal dual variables for the two CMDPs. It then follows from~\eqref{ineq:12} and~\eqref{ineq:13} that, 
\begin{align*}
	r^{\star} &\leq  \bar{r}(\pi^{(1)})+\lambda^{(2)} \left( c^{ub}- \bar{c}(\pi^{(1)}) \right),\\
	\mbox{ and }~~r^{\star} ( \hat{c}^{ub}) &\geq \bar{r}(\pi^{(1)})+  \lambda^{(2)} \left(\hat{c}^{ub}- \bar{c}(\pi^{(1)}) \right).
\end{align*}
Subtracting the second inequality from the first yields
\begin{align*}
	r^{\star} - r^{\star} (c^{ub}) &\leq \lambda^{(2)}\left( c^{ub} - \hat{c}^{ub} \right)\\
	&\leq \left(\left\{c^{ub} - \hat{c}^{ub}\right\} \right) \left( \lambda^{(2)}\right)\\
	&\leq  \left( \left\{c^{ub} - \hat{c}^{ub} \right\} \right) \frac{\hat{\eta}}{\eta},
\end{align*}
where the last inequality follows from Lemma~\ref{lemma:lagrange_ub}. This completes the proof. $\square$

\bibliographystyle{plain}
\bibliography{ref_file}

\end{document}